\documentclass[12pt]{article}
\usepackage[utf8]{inputenc}
\usepackage{amsmath}
\newcommand{\N}{\mathbb{N}}
\newcommand{\R}{\mathbb{R}}
\newcommand{\Z}{\mathbb{Z}}
\newcommand{\Q}{\mathbb{Q}}
\newcommand{\e}{\varepsilon}
\usepackage{relsize}
\usepackage{amsthm}
\usepackage{amssymb}

\usepackage[margin=1in]{geometry}

\newtheorem{theorem}{Theorem}[section]
\newtheorem{lemma}[theorem]{Lemma}
\newtheorem{proposition}[theorem]{Proposition}

\newtheorem{problem}{Problem}
\numberwithin{equation}{section}
\numberwithin{corollary}{section}
\numberwithin{definition}{section}
\numberwithin{problem}{section}
\usepackage{ccfonts}
\usepackage[T1]{fontenc}
\DeclareFontSeriesDefault[rm]{bf}{sbc}
\usepackage{bm}
\usepackage{stmaryrd}
\usepackage[parfill]{parskip}
\usepackage{titlesec}
\titleformat{\section}{\centering\normalfont\scshape}{\thesection}{1em}{}
\titleformat{\subsection}{\centering\normalfont\scshape}{\thesubsection}{1em}{}
\title{Dense Forests Constructed from Grids}
\author{Victor Shirandami\footnote{Department of Mathematics, University of Manchester, Manchester, United Kingdom \\ Email: victor.shirandami@postgrad.manchester.ac.uk}}
\date{}

\newcommand{\GLn}{\text{\normalfont GL}_n(\R)}

\newcommand{\s}{\mathbb{S}}

\newcommand{\Pd}{\mathbb{P}^d(\R)}

\usepackage{hyperref}
\hypersetup{
    colorlinks=true,
    linkcolor=blue,
    filecolor=magenta,      
    urlcolor=cyan,
    pdftitle={Overleaf Example},
    pdfpagemode=FullScreen,
    }
\makeatletter
\begin{document}

\maketitle
\begin{abstract}
    A dense forest is a set $F \subset \R^n$ with the property that for all $\e > 0$ there exists a number $V(\e) > 0$ such that all line segments of length $V(\e)$ are $\e$-close to a point in $F$. The function $V$ is called a visibility function of $F$. In this paper we study dense forests constructed from finite unions of translated lattices (grids). First, we provide a necessary and sufficient condition for a finite union of grids to be a dense forest in terms of the irrationality properties of the matrices defining them. This answers a question raised by Adiceam, Solomon, and Weiss (2022). To complement this, we further show that such sets generically admit effective visibility bounds in the following sense: for all $\eta > 0$, there exists a $k \in \N$ such that almost all unions of $k$ grids are dense forests admitting a visibility function $V(\e) \ll \e^{-(n-1) -\eta}$. This is arbitrarily close to optimal in the sense that if a finite union of grids admits a visibility function $V$, then this function necessarily satisfies $V(\e) \gg \e^{-(n-1)}$. One of the main novelties of this work is that the notion of `almost all' is considered with respect to several underlying measures, which are defined according to the Iwasawa decomposition of the matrices used to define the grids. In this respect, the results obtained here vastly extend those of Adiceam, Solomon, and Weiss (2022) who provided similar effective visibility bounds for a particular family of generic unimodular lattices.
\end{abstract}

\section{Introduction}\label{introduction}
A dense forest $F \subset \R^n$ is a set that admits a function $V: \R_{>0} \to \R_{>0}$, called a \emph{Visibility Function}, such that for all sufficiently small $\e > 0$, all line segments in $\R^n$ of length $V(\e)$ intersect with $\bigcup_{\bm f \in F} B_2 (\bm f, \e)$, where $B_2(\bm f, \e)$ denotes the open Euclidean ball of radius $\e$ centred at $\bm f$. This is saying that $F$ is uniformly close to all sufficiently long line segments. Sets with this property are considered with some further restrictions. First, one may stipulate that $F$ has finite density, in the sense that
\begin{equation*}
    \limsup_{T \to \infty} \frac{\#(F \cap B_2(\bm 0, T))}{T^n} < \infty.
\end{equation*}
Here, the symbol $\#$ denotes set cardinality. A significantly stronger restriction is uniform discreteness. This is when there is a uniform lower bound on the distance between any two distinct points in $F$. Clearly uniform discreteness implies finite density, but the converse is false. It is a simple exercise to show that if $F$ is both a dense forest and of finite density, it can only admit visibility functions such that $V(\e) \gg \e^{-(n-1)}$. Dense forests are closely related to the Danzer Problem, which concerns whether a set $D \subset \R^n$ that intersects every convex set of volume 1, called a Danzer set, exists when requiring that it has finite density. It is known that for $n=2$, there exists a Danzer set of finite density if and only if there exists a dense forest of finite density with visibility $V(\e) \ll \e^{-1}$ \cite{SW}. However, Bambah and Woods \cite{BW} have shown that for general $n$, Danzer sets constructed from a finite union of translated lattices cannot have finite density.

The first result on the construction of dense forests was due to Peres. This can be found in a work on the question of rectifiability of curves by Bishop \cite{Bishop}, which in fact led him to pose a version\footnote{He allows for the centres of the `trees' (points in the dense forest) to move as $\e>0$ varies. This is a much weaker version of what is considered a dense forest here, although the Peres construction can be easily adapted to conform to the canonical definition of a dense forest used in this paper.} of the dense forest problem. It is a planar dense forest constructed from a finite union of lattices with a visibility bound $V(\e) \ll \e^{-4}$. Although quite far from the ideal bound of $O(\e^{-1})$, this construction has the advantage of being fully deterministic. The visibility bound of this construction was improved to $O(\e^{-3})$ by Adiceam, Solomon, and Weiss \cite{ASW}. They also proved the existence of dense forests constructed from finite unions of translated lattices admitting a visibility function $V(\e) \ll_{\eta,n} \e^{-(n-1)-\eta}$ for all $\eta>0$. This was done by formulating a generalization of the Peres construction depending on a vector $\bm \Theta \in \R^{n-1}$ and then proving the visibility bound for almost all $\bm \Theta \in \R^{n-1}$ in the Lebesgue sense. From here-on, a translate of a lattice shall be referred to as a grid. The best known visibility bound for a deterministic construction of a dense forest is due to Tsokanos \cite{T}, namely

$$ V(\e) \ll_\eta \e^{-(n-1)} \log(\e^{-1}) \left(\log\log(\e^{-1})\right)^{1+\eta}$$

for all $\eta > 0$. However, this is not uniformly discrete. It also has the peculiar property that the set depends on $V$ via the choice of $\eta > 0$. The best known bound on visibility for a uniformly discrete dense forest (in the planar case) is given by Alon \cite{A}:

$$ V(\e) \ll \e^{-1} \exp\left\{C\sqrt{\log(\e^{-1})}\right\} $$

for some $C>0$. It is also claimed in this paper that similar constructions may be made in general dimensions. This construction, however, is non-deterministic. Solomon and Weiss \cite{SW} also provide uniformly discrete constructions but without effective visibility bound. The only known deterministic construction of a uniformly discrete dense forest with effective visibility bound is given in \cite{ASW} for the planar case, with visibility \nopagebreak $V(\e) \ll_\eta \e^{-5-\eta}$ for all $\eta > 0$, this also being a finite union of grids. Thus, the state of the art is fractured between the three properties: slow growth rate of the visibility function ($V(\e)$ `close' to $\e^{-(n-1)}$); uniform discreteness; deterministic construction.

This paper examines the construction of dense forests from finite unions of grids, both from a deterministic and a metrical perspective. This work was prompted by the following problem posed in \cite[\S 8, (2)]{ASW}.
\begin{problem}[Adiceam, Solomon, Weiss, 2022]\label{Honeycomb prob}
Define the \emph{Honey-Comb} lattice $\cal L$ by
\begin{equation}
    \mathcal{L} := \text{\normalfont span}_\Z \left\{ \binom{1}{0} , \binom{1/2}{\sqrt{3}/2}\right\} \subset \R^2.
\end{equation}
Can a union of $k \geq 2$ translated and rotated  copies of the Honey-Comb lattice be a dense forest? If so, what is the smallest allowed value of $k$?
\end{problem}
This is an interesting set to study because the structure of graphene is a naturally occurring manifestation of this lattice; it can be represented as a union of two Honeycomb Lattices. An active area of scientific research is into the physical properties of what is called twisted bi-layer graphene, composed of two layers of graphene, one on top of the another (for more on this topic see \cite{graphene1, graphene2} and the references therein).
This problem can be generalized to the following.
\begin{problem}\label{grids problem}
Given a collection of grids $G_1, \dots , G_k \subset \R^n$ does their union constitute a dense forest, and if so, what is the best possible visibility function it admits?
\end{problem}
A lattice $\Lambda \subset \R^n$ is associated with a matrix $M \in \text{GL}_n(\R)$, up to right multiplication by an element of $\text{SL}_n(\Z)$, such that $\Lambda = M \cdot \Z^n$. It turns out, however, that the dense forest property is sensitive only to a much coarser notion of equivalence. Namely, given $k \in \N$, it suffices to work in the space
\begin{equation*}
    \mathcal{S}_n^k := \big( \R^* \backslash \GLn / \text{GL}_n(\Q) \big)^k
\end{equation*}
where $\R^*$ denotes the multiplicative group of non-zero real numbers naturally identified with the group of homothetic matrices. This is stated more precisely in the following theorem, which further provides a necessary and sufficient condition for a finite union of grids to be a dense forest in terms of the irrationality properties of the matrices defining it. Let $\pi_k : \GLn^k \to \mathcal{S}_n^k$ denote the canonical projection.
\begin{theorem}\label{grids theorem}
Let $k\in \N_{\geq2}$ and $(M_1, \dots, M_k) \in \GLn^k$. The following are equivalent.
\begin{enumerate}
    \item There do not exist vectors $\bm v_1, \dots, \bm v_k \in \R^n$ each with rationally dependent components such that 
    \begin{equation*}
        M_1 \bm v_1 = \dots = M_k \bm v_k.
    \end{equation*}
    \item For all $\bm g_1, \dots ,\bm g_k \in \R^n$, and $(M'_1, \dots, M'_k) \in \GLn^k$ such that $\pi_k(M_1', \dots, M_k') = \pi_k(M_1, \dots, M_k)$, the set
    \begin{equation*}
        \bigcup_{i=1}^k (M_i' \cdot \Z^n + \bm g_i).
    \end{equation*}
    is a dense forest.
\end{enumerate}
\end{theorem}

This result establishes the dense forest property for a given union of grids without providing a visibility bound. With regards to Problem \ref{Honeycomb prob}, it shows that an explicit construction may be made in the best case $k = 2$.

The following result complements Theorem \ref{grids theorem}, proving that `almost all' unions of grids are dense forests with visibility arbitrarily close to optimal as the number of grids tends to infinity. The notation 
$$d:=n-1$$
is used in the below, and is adopted throughout the rest of this paper.

\begin{theorem}\label{metrical theorem}
Let $\mu$ be the Haar measure on $\text{\normalfont SO}(d+1)$, and let $\mu_k$ be its $k$-th product measure. Assume $k > d^2$. Given $M_1, \dots, M_k \in \emph{GL}_{d+1}(\R)$ and $\delta > 0$, for $\mu_k$-almost all $(R_1, \dots,R_k) \in \text{\normalfont SO}(d+1)^k$, and all $\bm g_1, \dots, \bm g_k \in \R^{d+1}$, the set
\begin{equation*}
    F = \bigcup_{i=1}^k \left( R_i M_i \cdot \Z^{d+1} + \bm g_i \right)
\end{equation*}
is a dense forest admitting a visibility function $V$ such that
\begin{equation*}
    V(\e) \ll_{d,\delta} \e^{-d-\sigma_d(k) - \delta},
\end{equation*}
where
\begin{equation*}
    \sigma_d(k) := \frac{d^2(d+1)}{k-d^2}\cdotp
\end{equation*}
\end{theorem}
The process of left multiplication of a matrix $M$ by a rotation matrix $R$ can be viewed as varying the Iwasawa decomposition of the matrix, given by 
\begin{equation*}
    M = RDT,
\end{equation*}
where $R \in \text{SO}(d+1)$, $D$ is a diagonal matrix, and $T$ is an upper triangular unipotent matrix. One may consider varying the other components of the Iwasawa decomposition of the matrices $M_1,\dots,M_k$. However, modifying the diagonal parts of these decompositions cannot generally lead to a dense forest. For example, if all the matrices are the identity matrix, any alteration of their diagonal parts yields a set of matrices corresponding to a union of grids that misses entire lines aligned with the coordinate axes. On the other hand, varying the upper triangular parts of the Iwasawa decompositions works similarly to the rotation parts, but with one caveat: it does not produce sets uniformly close to all line segments in the sense given by the definition of a dense forest. Instead, for a fixed angle $\theta\in(0,\pi/2)$, it produces sets uniformly close to line segments that make an angle strictly less than $\pi/2-\theta$ with the $(d+1)$-th axis. This does not forbid the construction dense forests; a union of $d+1$ rotated copies of the resulting set $F$ --- specifically, the set:
\begin{equation}\label{rotated sets}
    \bigcup_{i=1}^{d+1} Y_i F
\end{equation}
where, $Y_1, \dots, Y_{d+1} \in \text{SO}(d+1)$ are chosen such that $Y_i \bm e_{d+1} = \bm e_i$ --- produces a dense forest. The corresponding result is recorded in the following statement, where, given $\bm x \in \R^n$, one lets
\begin{equation*}
    T(\bm x) := 
    \begin{pmatrix}
    1 & 0 & \dots & 0& x_1 \\
    0 & 1 & \dots & 0 & x_2 \\
    \vdots & \vdots &\ddots & \vdots & \vdots \\
    0 & 0 & \dots & 1 & x_d \\
    0 & 0 & \dots & 0 & 1
    \end{pmatrix}.
\end{equation*}
\begin{theorem}\label{metrical theorem'}
Let $\lambda$ denote the Lebesgue measure on $\R^d$, and let $\lambda_k$ be its $k$-th product measure. Assume $k > d^2$. Given $M_1, \dots, M_k \in \emph{GL}_{d+1}(\R )$ and $\delta > 0$, for $\lambda_k$-almost all $(\bm x_1, \dots, \bm x_k) \in \R^{d \times k}$, and all $\bm g_1, \dots, \bm g_k \in \R^{d+1}$, the set defined in (\ref{rotated sets}) with
\begin{equation*}
    F = \bigcup_{i=1}^k \left(M_i T(\bm x_i) \cdot \Z^{d+1} + \bm g_i \right)
\end{equation*}
is a dense forest admitting a visibility function $V$ such that
\begin{equation*}
    V(\e) \ll_{d,\delta} \e^{-d-\sigma_d(k) - \delta},
\end{equation*}
where
\begin{equation*}
    \sigma_d(k) := \frac{d^2(d+1)}{k-d^2}\cdotp
\end{equation*}
\end{theorem}
In the above result, only the subset $\mathcal{T} = \{T(\bm x) : \bm x \in \R^d\}$ of the full space of upper triangular unipotent matrices is required. This is because the image of the group of matrices $\mathcal{T}$, acting on the line $L$ through $\bm 0$ parallel to the $(d+1)$-th axis, is the set of all lines through $\bm 0$ that make an angle strictly less than $\pi/2$ with $L$. On the other hand, the group of matrices $\text{SO}(d+1)$ acting on $L$ is the set of \emph{all} lines through $\bm 0$. Hence, Theorem \ref{metrical theorem} directly produces a dense forest, while Theorem \ref{metrical theorem'} requires taking a union of rotated copies of $F$ to obtain a dense forest. The proof of Theorem \ref{metrical theorem'} is nearly identical to that of Theorem \ref{metrical theorem}, with only minor modifications to the arguments, and is therefore not presented here.

The proof of Theorem \ref{metrical theorem} arises from a focused study of density properties of linear flows on the torus, resulting in Proposition \ref{LFT4}, which is of general utility. It is necessary to define some notions before stating it.

Given an $\bm x \in \R^{d+1} \setminus \{\bm 0\}$, let $[\bm x] \in \mathbb{P}^d(\R)$ denote the line $\{\lambda \bm x : \lambda \in \R\}$. Given two lines $L, L' \in \mathbb{P}^d(\R)$, let 
\begin{equation}\label{projective distance}
    \psi (L, L') := \sin \varphi
\end{equation}
where $\varphi \in [0,\pi/2]$ is the angle between $L$ and $L'$. Call this the \emph{projective distance} between $L$ and $L'$.

\begin{proposition}\label{LFT4}
Let $\delta >0$, $\bm u \in \s^d$, and $S \in \N$ such that $S > (d+1)^{d/2}$. If the linear flow
\begin{equation*}
    \left\{t \bm u \mod 1 : t \in \left[-\sqrt{d+1} \; S, \sqrt{d+1} \; S \right] \right\}
\end{equation*}
is not $\delta$-dense in $[0,1)^{d+1}$ with respect to the supremum norm, then there exists a $\bm q \in \Z^{d+1} \setminus \{\bm 0 \}$ such that
\begin{equation*}
    ||\bm q||_\infty < \frac{S^{1-1/d}}{\delta}, \quad\quad \text{and}\quad\quad
    \psi \left( [\bm u], [\bm q] \right) < \frac{d+1}{||\bm q||_\infty S^{1/d}}\cdotp
\end{equation*}
\end{proposition}
There seems to be few results in the literature on this type of question \cite{Abed, DF}. The above can be used to construct vectors $\bm u \in \s^d$ which have optimal filling times, that is: a filling time $T < C \delta^{-d}$ for some $C > 0$.

\paragraph{Structure of the paper.} Theorem \ref{grids theorem} is proved in \S \ref{DFP}, Proposition \ref{LFT4} is proved in \S \ref{LFT}, and Theorem \ref{metrical theorem} is proved in \S \ref{MT}.

\paragraph{Acknowledgements} The Author is grateful to Faustin Adiceam for his invaluable guidance throughout this project. The support of the Heilbronn Institute of Mathematical Research through the UKRI grant: Additional Funding Programme for Mathematical Sciences (EP/V521917/1) is gratefully acknowledged.

\paragraph{Notation} Some notations shall be defined here that will remain constant throughout the paper. Several other notations will be defined later where the proper context has been provided.
\begin{itemize}
    \item An integer $n \geq 2$ shall be reserved for the dimension of the ambient space and is considered as fixed. As stated in the introduction, $d \geq 1$ is fixed to $d:= n-1$.
    \item The standard basis of $\R^n$ shall be denoted by the vectors $\bm e_1, \dots, \bm e_n$ and the components $x_1, \dots, x_n$ of a vector $\bm x \in \R^n$ are with respect to this basis.
    \item Given $\bm x \in \R^n$, let $||\bm x||_2$ denote its Euclidean norm, $||\bm x||_\infty$ its supremum norm, and $||\bm x||$ its distance to $\Z^n$ with respect to the supremum norm, that is:
    \begin{equation*}
        ||\bm x|| := \min \{||\bm x - \bm q||_\infty :  \bm q \in \Z^n \}.
    \end{equation*}
    \item Given $\bm x \in \R^n$ and $r > 0$, let $B_2(\bm x, r)$ and $B_\infty (\bm x, r)$ denote the open balls centred at $\bm x$ and of radius $r$ with respect to the Euclidean and supremum norms, respectively. Let $\bar B_2 (\bm x, r)$, $\bar B_\infty( \bm x, r)$ denote their respective closures.
    \item Given real numbers $a < b$, let
    \begin{equation*}
        \llbracket a, b \rrbracket := [a,b] \cap \Z, \quad\quad \text{and} \quad \quad \llparenthesis a, b \rrparenthesis := (a,b) \cap \Z.
    \end{equation*}
    \item Subscripts are given to Vinogradov and Big O symbols when the implied constant in the expression is dependent on some given parameters, e.g., $f(x) \ll_a g(x)$ means that the implied constant in the Vinogradov symbol depends on a given parameter $a$. 
\end{itemize}

\section{Necessary and Sufficient Condition for a Union of Grids to be a Dense Forest}\label{DFP}
Given $\bm a \in \R^n$, $\bm b \in \s^{n-1}$, $l>0$, let
\begin{equation}\label{GPDF notation 1}
    L(\bm a, \bm b, l) := \{\bm a + \lambda \bm b : - l \leq \lambda \leq l\},
\end{equation}
that is: $L(\bm a, \bm b, l)$ is the line segment centred at $\bm a$, parallel to $\bm b$, and of length $2l$. For $S \subset \R^n$ and $\e>0$ define the \emph{directional visibility} function by
\begin{equation}\label{GPDF noation 2}
    \phi_\e[S](\bm a, \bm b) := \inf \left\{ l > 0: L(\bm a, \bm b, l) \cap \left( \bigcup_{\bm s \in S} B_2(\bm s, \e) \right) \neq \varnothing \right\},
\end{equation}
where $(\bm a, \bm b)$ takes values in $\R^n \times \s^{n-1}$. It is clear that $S$ is a dense forest if and only if for all $\e > 0$, the function $\phi_\e[S]$ is uniformly bounded on $\R^n \times \s^{n-1}$. In this section Theorem \ref{grids theorem} is proved. First, it is argued that Theorem \ref{grids theorem} follows from the equivalence of the following two statements for a given $(M_1, \dots, M_k) \in \GLn^k$:
\begin{enumerate}
    \item[\textbf{A}.] For all $\bm b \in \s^{n-1}$ at least one of the vectors
    \begin{equation*}
        M_1^{-1} \bm b, \dots, M_k^{-1} \bm b
    \end{equation*}
    has rationally independent components.
    \item[\textbf{B}.] For all $\bm g_1, \dots, \bm g_k \in \R^n$, the set
    \begin{equation*}
        \bigcup_{i=1}^k (M_i \cdot \Z^n + \bm g_i)
    \end{equation*}
    is a dense forest.
    \end{enumerate}
Since \textbf{A} clearly depends only on the coset in $\mathcal{S}_n^k$ to which $(M, \dots, M_k)$ belongs, it suffices to prove that \textbf{A} is equivalent to statement \emph{1} of Theorem \ref{grids theorem}. To see this, note that \textbf{A} is false if and only if there exist $\bm b  \in \s^{n-1}$ and vectors $\bm v_1, \dots, \bm v_k$ each with rationally dependent components such that $\bm v_i = M_i^{-1} \bm b$ for all $1 \leq i \leq k$. That is: there exist $\bm v_1, \dots, \bm v_k$ each with rationally dependent components such that
\begin{equation}\label{condition}
    M_1 \bm v_1 = \dots  = M_k \bm v_k \in \s^{n-1}.
\end{equation}
Note that the stipulation $M_i \bm v_i \in \s^{n-1}$ ($1 \leq i \leq k$) may be lifted, since if this is not the case one may map $\bm v_i \mapsto \bm v_i/||M_1 \bm v_1||_2$ for each $1 \leq i \leq k$. This is the negation of statement \emph{1} of Theorem \ref{grids theorem}. 

The goal is now to prove the equivalence of \textbf{A} and \textbf{B}. Let $(M_1, \dots, M_k) \in \text{GL}_n(\R)^k$ and let $\bm g_1, \dots, \bm g_k \in \R^n$. Set $G_i = M_i \cdot \Z^n + \bm g_i$ and
\begin{equation*}
    F = \bigcup_{i=1}^k G_i.
\end{equation*}
The directional visibility evaluates to
\begin{equation*}
    \phi_\e[F](\bm x, \bm b) = \min_{1 \leq i \leq k} \phi_\e[G_i](\bm x, \bm b),
\end{equation*}
for $(\bm x, \bm b) \in \R^n \times \s^{n-1}$, where the minimum is taken over all those $G_i$ ($1 \leq i \leq k)$ for which $\phi_\e[G_i](\bm x, \bm b)$ is well-defined. Fix a $\bm b \in \s^{n-1}$. Here, given $\e>0$, $\bm x \in \R^n$, $1 \leq i \leq k$, the quantity $\phi_\e[G_i](\bm x, \bm b)$ is well-defined if and only if there exists  a $\bm q \in \Z^n$ such that
\begin{equation}\label{DFP equation 1}
    \inf_{t \in \R} || M_i \bm q  + \bm g_i - (\bm x + t \bm b)||_2 < \e.
\end{equation}
Consider the linear flow:
\begin{equation}\label{DFP equation 2}
    \{t M_i^{-1}\bm b \mod 1 : t \in \R\}.
\end{equation}
It is well-known that the closure of (\ref{DFP equation 2}) is $[0,1)^n$ when $M_i^{-1} \bm b$ has rationally independent components, and a proper rational subspace otherwise (cf.\cite[Chap 9]{KN}). This implies that if (\ref{DFP equation 2}) is not dense in $[0,1)^n$, then for almost all $\bm x \in \R^n$ there exists an $\e>0$ where (\ref{DFP equation 1}) fails. Thus, if none of the sets (\ref{DFP equation 2}) for $1 \leq i \leq k$ are dense in $[0,1)^n$, then for almost all $\bm x \in \R^n$, $\phi_\e[F](\bm x, \bm b)$ is not well defined for some $\e>0$ (depending on $\bm x$). This yields the following.
\begin{lemma}\label{lemma DFP1}
    Let $\bm b \in \s^{n-1}$. The following are equivalent.
    \begin{enumerate}
        \item $\phi_\e[F](\bm x, \bm b)$ is well defined for all $\e>0$, and all $\bm x \in \R^n$.
        \item At least one of $\phi_\e[G_i](\bm x, \bm b)$ $(1 \leq i \leq k)$ is well-defined for all $\e>0$, and all $\bm x \in \R^n$.
        \item At least one of the vectors 
        \begin{equation*}
            M_1^{-1} \bm b, \dots, M_k^{-1} \bm b
        \end{equation*}
        has rationally independent components.
    \end{enumerate}
\end{lemma}
This implies that to prove the equivalence of \textbf{A} and \textbf{B}, it suffices to prove that for all $\e>0$, $\phi_\e[F]$ is well defined on all of $\R^n \times \s^{n-1}$ if and only if, for all $\e>0$, $\phi_\e[F]$ is \emph{uniformly bounded} on all of $\R^n \times \s^{n-1}$.

Let $\rho_i$ be the covering radius of $G_i$ ($ 1 \leq i \leq k$) with respect to the Euclidean norm, and set $\rho := \max_{1 \leq i \leq k} \rho_i$. Since each grid is periodic, for each $\bm x \in \R^n$ and $1 \leq i \leq k$, there exists an $\bm x_i \in \bar B_2(\bm 0, \rho)$ such that
\begin{equation*}
    \phi_\e[G_i](\bm x, \bm b) = \phi_\e[G_i](\bm x_i, \bm b)
\end{equation*}
for all $\bm b \in \s^{n-1}$. Given $\bm y_1, \dots, \bm y_k \in \bar B_2(\bm 0, \rho)$, $\bm b \in \s^{n-1}$, $\e>0$, let
\begin{equation*}
    \xi_\e(\bm y_1, \dots, \bm y_k; \bm b) := \min_{1 \leq i \leq k} \phi_\e[G_i](\bm y_i, \bm b).
\end{equation*}
By Lemma \ref{lemma DFP1}, $\xi_\e$ well defined on all of $\Omega := \bar B_2(\bm 0, \rho)^k \times \s^{n-1}$ for all $\e >0$ if and only if $\phi_\e[F]$ is well defined on all of $\R^n \times \s^{n-1}$ for all $\e>0$. Clearly, if $\xi_\e$ uniformly bounded on $\Omega$ for all $\e > 0$, then $\phi_\e[F]$ is uniformly bounded on $\R^n \times \s^{n-1}$ for all $\e > 0$. Thus it suffices to show that if $\xi_\e$ is well defined on $\Omega$ for all $\e>0$, then $\xi_\e$ is uniformly bounded on $\Omega$ for all $\e > 0$.

By elementary geometric arguments it easily seen that for an arbitrary set $S \subset\R^n$ and $(\bm a, \bm b) \in \R^n \times \s^{n-1}$,  if $\phi_\e[S](\bm a, \bm b)$ is well defined, then $\phi_\e[S]$ is uniformly bounded on an open neighborhood of $(\bm a, \bm b)$. Thus, for each $\e>0$, and each $P \in \Omega$, there exists an open neighborhood of $P$, say $U_\e(P)$, where $\xi_\e$ is uniformly bounded. Let $\lambda_\e(P)$ be this upper bound for $\xi_\e$ on $U_\e(P)$. The set
\begin{equation*}
    \{U_\e(P) : P \in \Omega\}
\end{equation*}
is an open cover of $\Omega$. Since $\Omega$ is compact, there exists a finite sub-cover, say,
\begin{equation*}
    \{U_\e(P_i) : 1 \leq i \leq N\},
\end{equation*}
where $N \in \N$ and $P_i \in \Omega$ for all $1 \leq i \leq N$. Therefore $\xi_\e$ is uniformly bounded above by $\max_{1 \leq i \leq N} \lambda_\e(P_i)$ on $\Omega$. This completes the proof of Theorem \ref{grids theorem}.\hfill $\Box$

The compactness argument at the end of the proof is a crucial step that precludes an explicit visibility bound.

\section{Filling Times for Linear Flows on The Torus}\label{LFT}
In contrast to the previous section, which only required determining whether a linear flow is dense in the torus, the metrical theory necessitates quantitative bounds for the filling time of linear flows on the torus. This is the focus of Proposition \ref{LFT4}, which will be proven in this section.

Given $\bm u \in \s^{n-1}$ and $\delta > 0$, the filling time of a linear flow
\begin{equation}\label{flow set}
    \Delta_T(\bm u) := \{ t \bm u \mod 1 : t \in [-T,T]\}
\end{equation}
is the greatest lower bound on the set of $T>0$ such that (\ref{flow set}) is $\delta$-dense in $[0,1)^n$. It shall be more convenient to define $\delta$-density with respect to the supremum norm. Upon application to the construction of dense forests in subsequent sections this will have an effect which modifies the visibility functions by at most a constant factor. Assume that $||\bm u||_\infty = |u_{d+1}|$ and define
\begin{equation*}
    \Sigma_T (\bm u) := \left\{m \left(\frac{u_1}{u_{d+1}}, \dots, \frac{u_d}{u_{d+1}}\right)^T \mod 1: m \in \llbracket -T, T \rrbracket \right\}.
    \end{equation*}
The following lemma transfers the continuous flow problem to that of discrete flows.
\begin{lemma}\label{LFT2}
    Let $\delta>0$, $S \in \N$, and let $\bm u \in \s^d$ such that $||\bm u||_\infty = |u_{d+1}|$. If $\Sigma_S(\bm u)$ is $\delta$-dense in $[0,1)^d$, then $\Delta_{\sqrt{d+1}S}(\bm u)$ is $\delta$-dense in $[0,1)^{d+1}$.
\end{lemma}
\begin{proof}
Fix $\delta > 0$, $\bm u \in \s^d$, and identify the torus with $[0,1)^{d+1}$. For each $\sigma \in [0,1)$ set
\begin{equation*}
    F_\sigma := \{(x_1, \dots , x_d, \sigma)^T : x_1, \dots, x_d \in [0,1)\}.
\end{equation*}
It is claimed that, given $\sigma, \sigma' \in [0,1)$, if $T \in |u_{d+1}|^{-1} \N$ and if $\Delta_T(\bm u) \cap F_\sigma$ is $\delta$-dense in $F_\sigma$, then $\Delta_T(\bm u)\cap F_{\sigma'}$ is $\delta$-dense in $F_{\sigma'}$. Indeed, because $T \in |u_{d+1}|^{-1} \N$, the endpoints, $-T \bm u \mod 1$ and $T \bm u \mod1$, of the linear flow $\Delta_T(\bm u)$ lie on $F_0$. Thus, there exists a collection of parallel lines $L_1, \dots, L_s \subset \R^{d+1}$ for some $s \in \N$ such that $\Delta_T(\bm u) = \bigcup_{i=1}^s L_i \cap [0,1)^{d+1}$. Since the lines are all parallel, the set $\Delta_T(\bm u) \cap F_{\sigma'}$ is given by $\Delta_T(\bm u) \cap F_{\sigma}  + \bm v \mod 1$ for some $\bm v \in [0,1)^{d+1}$. This proves the claim.

Since the trajectory $\Delta_\infty (\bm u) := \{t \bm u \mod 1 : t \in \R\}$ intersects with $F_0$ at all times $t \in u_{d+1}^{-1} \Z$, the hypothesis that $\Sigma_S(\bm u)$ is $\delta$-dense in $[0,1)^d$ is equivalent to the assertion that $\Delta_{|u_{d+1}|^{-1}S}(\bm u)$ is $\delta$-dense in $F_0$. Therefore, by the above claim, $\Delta_{|u_{d+1}|^{-1}S}(\bm u)$ is $\delta$-dense in all $F_\sigma$ ($ \sigma \in [0,1)$), implying that it is $\delta$-dense in the cube $[0,1)^{d+1}$. Noting that $|u_{d+1}| \geq 1/\sqrt{d+1}$, it follows that $\Delta_{|u_{d+1}|^{-1} S} (\bm u) \subseteq \Delta_{\sqrt{d+1} S} (\bm u)$. This completes the proof.
\end{proof}
This provides the necessary context for the following classical transference theorem (cf.\cite[Chap V, Theorem VI]{Cas1957}) to be applied. In the below, the notation $\lfloor x \rfloor$ denotes the greatest integer not greater than $x \in \R$.
\begin{proposition}\label{transference theorem}
Let $L_1( \bm x), \dots , L_N (\bm x)$ be $N \in \N$ homogeneous linear forms in an integral variable $\bm x \in \Z^M$, where $M \in \N$. Given $C, X > 0$, if:
\begin{equation*}
    \max_{1 \leq i \leq N} ||L_i(\bm x)|| \geq C \quad \quad \forall \; \bm x \in \llparenthesis -X, X \rrparenthesis^M \setminus \{\bm 0\},
\end{equation*}
then for all $\bm \alpha \in [0,1)^N$, there exists an $\bm x \in \Z^M$ such that
\begin{equation*}
    \max_{1 \leq i \leq N} || L_i (\bm x) - \alpha_i|| \leq C', \quad \quad ||\bm x ||_\infty \leq X',
\end{equation*}
where
\begin{equation*}
    C' = \frac{1}{2}(h+1) C, \quad \quad X' = \frac{1}{2}(h+1) X, 
\end{equation*}
and where
\begin{equation*}
    h = \lfloor X^{-M} C^{-N} \rfloor.
\end{equation*}
\end{proposition}
This has the following consequence in the special case $N=d$, $M=1$.
\begin{lemma}\label{LFT3}
Let $\delta >0$, $S \in \N$, and let $\bm u \in \s^d$ such that $||\bm u||_\infty = |u_{d+1}|$. If
\begin{equation}\label{diophantine lemma 1 condition}
    \max_{1 \leq i \leq d} || m u_i/u_{d+1} || \geq S^{-1/d} \quad\quad \forall\; m \in \llparenthesis -\delta^{-1} S^{1-1/d}, \delta^{-1} S^{1-1/d} \rrparenthesis \setminus \{0\},
\end{equation}
then the linear flow
\begin{equation*}
    \Delta_{\sqrt{d+1}S}(\bm u) = \left\{t \bm u \mod 1 : t \in \left[-\sqrt{d+1} \; S, \sqrt{d+1} \; S \right] \right\}
\end{equation*}
is $\delta$-dense in $[0,1)^{d+1}$.
\end{lemma}
\begin{proof}
By Lemma \ref{LFT2}, it suffices to show, given the hypotheses, that $\Sigma_S (\bm u)$ is $\delta$-dense in $[0,1)^d$. Proposition \ref{transference theorem} implies that if, given $C,X >0$,
\begin{equation}\label{diophantine lemma 1 inequality}
    \max_{1 \leq i \leq d} || m u_i/u_{d+1} || \geq C \quad\quad \forall \; m \in \llparenthesis -X, X \rrparenthesis \setminus \{0\},
\end{equation}
then $\Sigma_{X'} (\bm u)$ is $C'$-dense in $[0,1)^d$, where
\begin{equation*}
    C' = \frac{1}{2}(h+1) C, \quad\quad X' = \frac{1}{2}(h+1) X,
\end{equation*}
and
\begin{equation*}
    h = \lfloor X^{-1} C^{-d} \rfloor.
\end{equation*}
By Dirichlet's theorem, there exists an $m \in \llbracket 1, X\rrparenthesis$ such that
\begin{equation*}
    \max_{1 \leq i \leq d} || m u_i/u_{d+1}|| \leq X^{-1/d}.
\end{equation*}
Thus, the condition (\ref{diophantine lemma 1 inequality}) may be satisfied only if $C \leq X^{-1/d}$. This implies that $h \geq 1$, whence
\begin{equation*}
    C' \leq X^{-1} C^{-(d-1)} \quad\quad \text{and}\quad\quad X' \leq C^{-d}.
\end{equation*}
Choose $C, X > 0$, such that
\begin{equation*}
    \delta = X^{-1} C^{-(d-1)} \quad\quad \text{and} \quad\quad S = C^{-d}.
\end{equation*}
Then, $C' \leq \delta$ and $X' \leq S$. Since $\Sigma_{X'} (\bm u) \subseteq \Sigma_S (\bm u)$, it follows that $\Sigma_S (\bm u)$ is $\delta$-dense in $[0,1)^d$. Expressing $C,X$ in terms of $\delta, S$ yields
\begin{equation*}
    C = S^{-1/d}, \quad\quad X = \delta^{-1} S^{1-1/d},
\end{equation*}
which completes the proof.
\end{proof}
The above provides enough machinery to prove Proposition \ref{LFT4}.
\begin{proof}[Proof of Proposition \ref{LFT4}]
Without loss of generality assume that $||\bm u||_\infty = |u_{d+1}|$, whence $|u_{d+1}| \geq 1/\sqrt{d+1}$. By Lemma \ref{LFT3}, if the linear flow
\begin{equation*}
    \left\{t \bm u \mod 1 : t \in \left[-\sqrt{d+1} \; S, \sqrt{d+1} \; S \right] \right\}
\end{equation*}
is not $\delta$-dense in $[0,1)^{d+1}$, then there exists an
\begin{equation*}
    m \in \llparenthesis - \delta^{-1} S^{1-1/d}, \delta^{-1} S^{1-1/d} \rrparenthesis \setminus \{0\}
\end{equation*}
such that
\begin{equation*}
    \max_{1 \leq i \leq d} || m u_i/u_{d+1} || < S^{-1/d}.
\end{equation*}
Writing $\bm {\tilde u} := (u_1, \dots, u_d)^T \in \R^d$, this implies that there exists a $\bm {\tilde q} \in \Z^d$ such that
\begin{equation*}
    ||m \bm{\tilde u} - u_{d+1} \bm{\tilde q}||_{\infty} < \frac{|u_{d+1}|}{S^{1/d}} \leq \frac{1}{S^{1/d}}\cdotp
\end{equation*}
Write $q_{d+1}$ for $m$, and let $\bm q = (\bm{\tilde q}^T, q_{d+1})^T \in \Z^{d+1}$. Clearly $\bm q \neq \bm 0$ since $q_{d+1}\neq 0$. It follows that
\begin{equation}\label{projective proposition inequlity 1}
    ||q_{d+1} \bm u - u_{d+1} \bm q ||_\infty < S^{-1/d}.
\end{equation}
Applying the triangle inequality yields
\begin{equation*}
    \Big| |q_{d+1}|\;||\bm u||_{\infty} - |u_{d+1}| \; ||\bm q||_{\infty}\Big| < S^{-1/d}. 
\end{equation*}
As $|u_{d+1}| = ||\bm u ||_\infty$, 
\begin{equation*}
    \Big| |q_{d+1}| - ||\bm q||_\infty \Big| < \frac{1}{|u_{d+1}| S^{1/d}} \leq \frac{\sqrt{d+1}}{S^{1/d}},
\end{equation*}
giving
\begin{equation*}
    |q_{d+1}| > ||\bm q||_\infty - \frac{\sqrt{d+1}}{S^{1/d}}\cdotp
\end{equation*}
Since $\frac{\sqrt{d+1}}{S^{1/d}} < 1$ for all $S > (d+1)^{d/2}$, and since $|q_{d+1}|, ||\bm q||_\infty \in \N$, this implies
\begin{equation}\label{projective proposition inequality 2}
||\bm q||_\infty = |q_{d+1}| < \frac{S^{1-1/d}}{\delta}\cdotp
\end{equation}
Moreover,
\begin{equation*}
    ||\bm q u_{d+1} ||_\infty = |u_{d+1}| \; ||\bm q||_\infty \geq \frac{||\bm q ||_\infty}{\sqrt{d+1}},
\end{equation*}
whence
\begin{equation}\label{projective proposition inequality 4}
    \min\Big\{||\bm u q_{d+1}||_2 \; ; \; ||\bm q u_{d+1} ||_2 \Big\} \geq \frac{||\bm q||_\infty}{\sqrt{d+1}}\cdotp
\end{equation}
By elementary trigonometry it is readily shown that, given $\bm a, \bm b \in \R^{d+1} \setminus \{\bm 0\}$ with angle $\varphi \in [0,\pi/2]$ between the lines they determine, the following inequality holds:
\begin{equation*}
    2 \min \left\{ ||\bm a||_2 ; ||\bm b||_2 \right\} \sin \frac{\varphi}{2} \leq ||\bm a - \bm b||_2.
\end{equation*}
This, and inequalities (\ref{projective proposition inequlity 1}) and (\ref{projective proposition inequality 4}) imply
\begin{align*}
    S^{-1/d} &> || \bm u q_{d+1} - u_{d+1} \bm q ||_{\infty} \\
    &\geq \frac{1}{\sqrt{d+1}} || \bm u q_{d+1} - u_{d+1} \bm q ||_2 \\
    &\geq \frac{2}{\sqrt{d+1}} \min \Big\{ ||\bm u q_{d+1}||_2 \; ; \; ||\bm q u_{d+1}||_2 \Big\} \sin \frac{\varphi}{2} \\
    &\geq \frac{2 ||\bm q||_\infty \sin \frac{\varphi}{2}}{d+1},
\end{align*}
where $\varphi \in [0,\pi/2]$ is the angle between $[\bm u q_{d+1}]$ and $[\bm q u_{d+1}]$. Noting that $2\sin\frac{\varphi}{2} \geq \sin \varphi$ for all $\varphi \in [0,\pi/2]$, the above yields
\begin{equation}\label{projective proposition inequlity 3}
    \sin \varphi < \frac{d+1}{||\bm q||_\infty S^{1/d}} \cdotp
\end{equation}
Inequalities (\ref{projective proposition inequality 2}) and (\ref{projective proposition inequlity 3}) establish the proposition.
\end{proof}
\section{The Metrical Theory}\label{MT}
Given $\bm v \in \s^d$ and $\eta \in (0,1)$, define a bi-spherical cap to be a set $\mathcal{K}(\bm v, \eta) \subset \s^d$ of the form:
\begin{equation*}
    \mathcal{K}(\bm v, \eta) = \{ \bm u \in \s^d : \psi([\bm u], [\bm v]) < \eta\},
\end{equation*}
where $\psi$ denotes the projective distance, as defined in (\ref{projective distance}). Call $[\bm v]$ the centre of $\mathcal{K}$, and $\eta$ the radius of $\mathcal{K}$. In the proceeding metrical arguments it will be useful to have an optimal covering of $\s^d$ by bi-spherical caps (in a suitable sense), which is provided by the following lemma.
\begin{lemma}\label{spherical caps lemma}
Given $\eta \in (0,1)$, the sphere $\s^d$ may be covered by $O_d(\eta^{-d})$ bi-spherical caps with radius $\eta$.
\end{lemma}
\begin{proof}
Follows from Theorem 6.8.1 of \cite{Bor2004}.
\end{proof}
Define the projective Lipschitz constant of an $M \in \text{GL}_{d+1}(\R)$ by
\begin{equation*}
    J_{\Pd}(M) := \sup \left\{ \frac{\psi([M\bm u], [M\bm v])}{\psi([\bm u], [\bm v])}:  \bm u, \bm v \in \s^d, [\bm u] \neq [\bm v]\right\}.
\end{equation*}
The following proposition shall be used in conjunction with the Borel-Cantelli Lemma to complete the proof of Theorem \ref{metrical theorem}. The proof proceeds by reformulating the problem into one about filling times of linear flows on the torus, allowing Proposition \ref{LFT4} to be applied.
\begin{proposition}\label{complicated proposition}
Let $f: \R_{>0} \to \R_{>0}$ be a monotonic increasing function such that $f(2\e)/f(\e) \ll 1$ as $\e \to 0^+$. Consider
\begin{equation*}
    F = \bigcup_{i=1}^k \left( R_i M_i \cdot \Z^{d+1} + \bm g_i\right),
\end{equation*}
where $R_i \in \text{\emph{SO}}(d+1)$,  $M_i \in \emph{GL}_{d+1}(\R)$ and $\bm g_i \in \R^{d+1}$ for all $1 \leq i \leq k$. For each $l \in \N$ let  $\mathcal{C}_l$ be a covering of $\s^d$ by bi-spherical caps of radius $2^{-l}/f(2^{-l})$. Write $\mathcal{D}_l$ for the set of centres of the bi-spherical caps in $\mathcal{C}_l$. If $F$ does not admit a visibility function $V: \R_{>0} \to \R_{>0}$ such that $V(\e) \ll f(\e)$ as $\e \to 0^+$, then for infinitely many $l \in \N$ there exists a point $\bm b \in \mathcal{D}_l$ and integer vectors $\bm q_1, \dots, \bm q_k \in \Z^{d+1} \setminus \{ \bm 0 \}$ such that for all $1 \leq i \leq k$\emph{:}
\begin{equation*}
    ||\bm q_i||_\infty \ll_{d,M_1, \dots, M_k} 2^l f(2^{-l})^{1-1/d}
\end{equation*}
and
\begin{equation*}
    \psi([\bm b], [R_i M_i \bm q_i]) \ll_{d, M_1, \dots, M_k} ||\bm q_i||^{-1}_\infty f(2^{-l})^{-1/d}.
\end{equation*}
\end{proposition}
\begin{proof}
Assume $F$ does not admit a visibility function $V: \R_{>0} \to \R_{>0}$ such that $V(\e) \ll f(\e)$ as $\e \to 0^+$. Since $f$ is monotonic increasing, and $f(2\e)/f(\e) \ll 1$ as $\e \to 0^+$, it is easily seen that it is sufficient to consider $\e \in \{2^{-l} : l \in \N\}$ in order to establish a visibility function up to a constant factor. Therefore, for infinitely many $l \in \N$, the inequality
\begin{equation*}
    \inf_{|t| \leq f(2^{-l})} \; \min_{1 \leq i \leq k} \; \min_{\bm q \in \Z^{d+1}} \; ||R_i M_i \bm q + \bm g_i - (\bm a + \bm b t)||_\infty < 2^{-l}
\end{equation*}
is insoluble for some $\bm a \in \R^{d+1}$, $\bm b \in \s^d$. It is also not difficult to show that it is sufficient to consider $\bm b \in \mathcal{D}_l$ in order to establish a visibility function up to a constant factor. Therefore, there exists a point $\bm b \in \mathcal{D}_l$ such that none of the linear flows
\begin{equation*}
    \left\{t(R_iM_i)^{-1} \bm b \mod 1 : t \in [-f(2^{-l}), f(2^{-l})] \right\} \quad\quad (1 \leq i \leq k)
\end{equation*}
are $\delta$-dense in $[0,1)^{d+1}$, where
\begin{equation*}
    \delta \ll_{d,M_1, \dots, M_k} 2^{-l}.
\end{equation*}
Applying Proposition \ref{LFT4} with
\begin{equation*}
    \bm u = \frac{(R_iM_i)^{-1} \bm b}{|| (M_i R_i)^{-1} \bm b||_2}, \quad \quad S = \frac{f(2^{-l})|| (M_i R_i)^{-1} \bm b||_2}{\sqrt{d+1}}
\end{equation*}
shows that there exist $\bm q_1, \dots, \bm q_k \in \Z^{d+1} \setminus \{\bm 0\}$ such that
\begin{equation}\label{another inequality}
    ||\bm q_i||_\infty \ll_{d,M_1, \dots, M_k} 2^l f(2^{-l})^{1-1/d}
\end{equation}
and
\begin{equation}\label{inequality}
    \psi([(R_i M_i)^{-1} \bm b], [\bm q_i]) \ll_{d, M_1, \dots, M_k} \frac{1}{||\bm q_i||_\infty f(2^{-l})^{1/d}},
\end{equation}
which yields
\begin{equation}\label{another inequality 2}
    \psi([\bm b], [R_i M_i \bm q_i]) \ll_{d, M_1, \dots, M_k} \frac{J_{\Pd}(M_i)}{||\bm q_i||_\infty f(2^{-l})^{1/d}} \ll_{d, M_1, \dots, M_k} \frac{1}{||\bm q_i||_\infty f(2^{-l})^{1/d}}\cdotp
\end{equation}
Inequalities (\ref{another inequality}) and (\ref{another inequality 2}) then establish the proposition.
\end{proof}
Recalling the assumptions in Theorem \ref{metrical theorem}, let $M_1, \dots , M_k \in \text{GL}_{d+1}(\R)$ be fixed, and set $\mu_k$ to be the $k$-th product of the Haar (probability) measure $\mu$ on $\text{SO}(d+1)$. For $\bm b \in \s^d$, $\bm q_1, \dots, \bm q_k  \in \Z^{d+1}$ and $\eta_1, \dots, \eta_k \in (0,1)$, set
\begin{align}\label{X}
    X(\bm b; \bm q_1, \dots, \bm q_k ; \eta_1, \dots, \eta_k) := \big\{& (R_1, \dots , R_k) \in \text{SO}(d+1)^k : \\ 
    &\psi([\bm b], [R_i M_i \bm q_i]) < \eta_i \;\; \forall \;1 \leq i \leq k \big\}. \nonumber
\end{align}
An estimate is needed for the measure of $X(\bm b; \bm q_1, \dots, \bm q_k; \eta_1, \dots, \eta_k)$.
\begin{lemma}\label{measure lemma}
Given arbitrary $\bm b \in \s^d$, $\bm q_1, \dots, \bm q_k  \in \Z^{d+1}$, and $\eta_1, \dots, \eta_k \in (0,1)$, the measure of $X(\bm b; \bm q_1, \dots, \bm q_k ; \eta_1, \dots, \eta_k)$, defined in (\ref{X}), is such that
\begin{equation*}
    \mu_k \big( X(\bm b; \bm q_1, \dots, \bm q_k ; \eta_1, \dots, \eta_k) \big) \ll_d (\eta_1 \dots \eta_k)^d.
\end{equation*}
\end{lemma}
\begin{proof}
The measure of this set can be written as a product:
\begin{align*}
    &\mu_k \big( X(\bm b; \bm q_1, \dots, \bm q_k ; \eta_1, \dots, \eta_k) \big) \\
    &= \prod_{i=1}^k \mu \big(\{R \in \text{SO}(d+1) : \psi([\bm b], [R M_i \bm q_i]) < \eta_i\} \big).
\end{align*}
For $\bm b, \bm c \in \s^d$, and $\eta \in (0,1)$, set 
\begin{equation*}
    Y(\bm b, \bm c, \eta) := \{R \in \text{SO}(d+1) : \psi([\bm b], [R \bm c]) < \eta\}.
\end{equation*}
It suffices to prove
\begin{equation}\label{estimate}
    \mu(Y(\bm b, \bm c, \eta)) \ll_d \eta^d
\end{equation}
to complete the proof. Let
\begin{equation*}
    \Gamma(\bm b, \bm c) := \{R \in \text{SO}(d+1) : R \bm b = \bm c\}.
\end{equation*}
Then
\begin{equation*}
    Y(\bm b, \bm c, \eta) = \bigcup_{\substack{\bm c' \in \s^d \\ \psi([\bm c], [\bm c']) < \eta}} \Gamma(\bm b, \bm c').
\end{equation*}
Given $A \subset \s^d$, let
\begin{equation*}
    m_{\bm b} (A) := \mu \left( \bigcup_{\bm c' \in A} \Gamma(\bm b, \bm c') \right).
\end{equation*}
This defines probability a measure on $\s^d$, and by the invariance of $\mu$ under $\text{SO}(d+1)$, it can be shown that $m_{\bm b}(RA) = m_{\bm b}(A)$ for any spherical cap $A \subset \s^d$, thereby uniquely defining $m_{\bm b}$ as the uniform measure on $\s^d$ \cite{C}. Therefore $\mu(Y(\bm b, \bm c, \eta))$ is given by the uniform measure of the bi-spherical cap $\mathcal{K}(\bm c, \eta)$.
This has a closed form (cf. \cite{sphericalcaps}) which yields the approximation $|\mathcal{K}(\bm c, \eta)| \ll_d \eta^d$, upon denoting by $|\cdot|$ the uniform measure on $\s^d$. This implies the estimate (\ref{estimate}), completing the proof.
\end{proof}
Let $f: \R_{>0} \to \R_{>0}$ be such that $f(2\e)/f(\e) \ll 1$ as $\e \to 0^{+}$. By Proposition \ref{complicated proposition}, there exist positive constants $U_1, U_2$ depending only on $d,M_1, \dots, M_k$ such that, if 
$$F = \bigcup_{i=1}^k ( R_i M_i \cdot \Z^{d+1} + \bm g_i)$$
does not admit a visibility function $V: \R_{>0} \to \R_{>0}$ obeying $V(\e) \ll f(\e)$ as $\e \to 0^+$, then $(R_1, \dots R_k) \in A_l$ for infinitely many $l \in \N$, where
\begin{equation*}
    A_l = \bigcup_{\bm b \in \mathcal{D}_l} \bigcup_{\substack{\bm q_i \in \Z^{d+1} \setminus \{\bm 0 \} \\ 0 < ||\bm q_i||_\infty < 2^l U_1 f(2^{-l})^{1-1/d} \\ 1 \leq i \leq k}} X(\bm b; \bm q_1,\dots,\bm q_k ; \eta_l (\bm q_1), \dots, \eta_l (\bm q_k)).
\end{equation*}
Here, $\mathcal{D}_l$ is the set of centres of bi-spherical caps in a covering $\mathcal{C}_l$ of $\s^d$ by $\ll_d (f(2^{-l})/2^{-l})^d$ bi-spherical caps of radius $2^{-l}/f(2^{-l})$ (such a covering is guaranteed by Lemma \ref{spherical caps lemma}), and 
\begin{equation*}
    \eta_l(\bm q) = U_2 ||\bm q ||_\infty^{-1} f(2^{-l})^{-1/d}.
\end{equation*}
In the following, the implied constants in the Vinogradov symbols depend only on $d, M_1, \dots$, $M_k$. By Lemma \ref{measure lemma}, given $l \in \N$, the set $A_l$ has measure
\begin{align*}
    \mu_k(A_l) &\ll \# \mathcal{D}_l \left(\sum_{\substack{\bm q_i \in \Z^{d+1} \setminus \{\bm 0 \} \\ 0 < ||\bm q_i||_\infty < 2^l U_1 f(2^{-l})^{1-1/d} \\ 1 \leq i \leq k}} \big( \eta_l(\bm q_1) \dots \eta_l(\bm q_k) \big)^d \right).
\end{align*}
Since $\# \mathcal{D}_l \ll (f(2^{-l})/2^{-l})^d$ and $\eta_l(\bm q) \ll ||\bm q||_\infty^{-1} f(2^{-l})^{-1/d}$, this implies
\begin{align*}
    \mu_k(A_l) &\ll \left( \frac{f(2^{-l})}{2^{-l}} \right)^d \left( \sum_{\substack{\bm q \in \Z^{d+1} \setminus \{\bm 0 \} \\ 0 < ||\bm q||_\infty < 2^l U_1 f(2^{-l})^{1-1/d}}} ||\bm q||_\infty^{-d} \right)^k f(2^{-l})^{-k} \\
    &\ll 2^{ld} f(2^{-l})^{d-k} \big(2^l f(2^{-l})^{1-1/d} \big)^k \\
    &= 2^{l(d+k)} f(2^{-l})^{d-k/d}.
\end{align*}
Here, the estimate
\begin{equation*}
    \sum_{\substack{\bm q \in \Z^{d+1}\setminus \{\bm 0\} \\ 0 < ||\bm q||_\infty < K}} ||\bm q||_\infty^{-d} \ll_d \int_{ 1 \leq ||\bm x||_2 < K} \frac{\text{d}^{d+1}\;\bm x}{||\bm x||_2^d} \ll_d K
\end{equation*}
is used between the first and second line of the above. Specialise the function $f$ to $f(\e) = \e^{-d-\lambda}$ for some $\lambda >0$ so that
\begin{align*}
    \mu_k(A_l) \ll 2^{l(d+k)} \cdot 2^{l(d-k/d)(d+\lambda)}
    = 2^{l(d + d^2 +\lambda d -k\lambda/d)}
    = 2^{l\left(d(1+\lambda+d) - k\lambda/d\right)}.
\end{align*}
Therefore,
\begin{equation*}
    \sum_{l=1}^\infty \mu_k(A_l) < \infty \quad\quad \text{whenever}\quad\quad d(1+\lambda+d) - k\lambda/d < 0.
\end{equation*}
Thus, by the Borel Cantelli Lemma, if
\begin{equation*}
    \lambda > \frac{d^2(d+1)}{k - d^2},
\end{equation*}
then
\begin{equation*}
    \mu_k \left( \limsup_{l \to \infty} A_l \right) = 0.
\end{equation*}
This completes the proof of Theorem \ref{metrical theorem}. \hfill $\Box$
\bibliographystyle{siam}
\bibliography{main}
\end{document}